\documentclass[reqno,b5paper]{amsart}

\usepackage{amsmath}
\usepackage{amssymb}
\usepackage{amsthm}
\usepackage{enumerate}
\usepackage[mathscr]{eucal}
\usepackage{eqlist}

\theoremstyle{plain}
\newtheorem{theorem}{Theorem}[section]

\newtheorem{lemma}{Lemma}[section]

\theoremstyle{definition}
\newtheorem{definition}{Definition}[section]
\newtheorem{remark}{\textnormal{\textbf{Remark}}}

\theoremstyle{remark}

\numberwithin{equation}{section}

\begin{document}

\title[2. The controllability structure]%
{On the internal approach to differential equations\\ 2. The controllability structure}

\author[Veronika Chrastinov\'a \and V\'aclav Tryhuk]%
{Veronika Chrastinov\'a* \and V\'aclav Tryhuk**}

\newcommand{\acr}{\newline\indent}

\address{\llap{*\,}Brno University of Technology\acr
Faculty of Civil Engineering\acr
Department of Mathematics\acr
Veve\v{r}\'{\i} 331/95, 602 00 Brno\acr
Czech Republic}
\email{chrastinova.v@fce.vutbr.cz}

\address{\llap{**\,}Brno University of Technology\acr
Faculty of Civil Engineering\acr
AdMaS Center\acr
Veve\v{r}\'{\i} 331/95, 602 00 Brno\acr
Czech Republic}
\email{tryhuk.v@fce.vutbr.cz}

\thanks{This paper was elaborated with the financial support of the European
Union's "Operational Programme Research and Development for
Innovations", No. CZ.1.05/2.1.00/03.0097, as an activity of the
regional Centre AdMaS "Advanced Materials, Structures and
Technologies".}



\begin{abstract}
The article concerns the geometrical theory of general systems $\Omega$ of partial differential equations in the \emph{absolute sense}, i.e., without any additional structure and subject to arbitrary change of variables in the widest possible meaning. The main result describes the composition series $\Omega^0\subset\Omega^1\subset\cdots\subset\Omega$ where $\Omega^k$ is the maximal system of differential equations "induced" by $\Omega$ such that the solution of $\Omega^k$ depends on arbitrary functions of $k$ independent variables (on constants if $k=0$). This is a~well--known result for the particular case of underdetermined systems of ordinary differential equations. Then $\Omega=\Omega^1$ and we have the composition series $\Omega^0\subset\Omega^1=\Omega$ where $\Omega^0$ involves all first integrals of $\Omega,$ therefore $\Omega^0$ is trivial (absent) in the controllable case. The general composition series $\Omega^0\subset\Omega^1\subset\cdots\subset\Omega$ may be regarded as a~"multidimensional" controllability structure for the partial differential equations.

Though the result is conceptually clear, it cannot be included into the common jet theory framework of differential equations. Quite other and genuinely coordinate--free approach is introduced.
\end{abstract}

\maketitle

\section{Preface}\label{sec1}
The origin of differential geometry rests on the investigation of surfaces firmly localized in the Euclidean space, this is the~\emph{external theory} by Euler. Subsequently the~\emph{internal theory} due to Riemann was investigated where the ambient Euclidean space disappears. The geometrical theory of differential equations is actually subject to analogous reconstruction.

Turning to more detail, let us recall the space $\mathbf M(m,n)$ supplied with (local) jet coordinates
\begin{equation}\label{t1.1}x_i,\ w^j_I\qquad (i=1,\ldots,n;\,j=1,\ldots,m;\,I=i_1\cdots i_r;\, r=0,1,\ldots\,)\end{equation}
where $x_i$ are independent variables, $w^j$ (empty $I=\emptyset)$ are dependent variables and $w^j_I$ stand for the derivatives
\[\frac{\partial^{i_1+\cdots+i_r}w^j}{\partial x_{i_1}\cdots\partial x_{i_r}}\qquad (i_1,\ldots,i_r=1,\ldots,n;\, j=1,\ldots,m;\, r=0,1,\ldots\,).\]
Roughly saying, $\mathbf M(m,n)$ is (locally) the infinite--order jet space of smooth cross-sections of the fibered space $\mathbb R^{m+n}\rightarrow\mathbb R^n.$ Differential equations are traditionally regarded as a~subspace $\mathbf M\subset\mathbf M(m,n)$ locally described by certain conditions
\begin{equation}\label{t1.2}D_{i_1}\cdots D_{i_r}f^k=0\qquad (i_1,\ldots,i_r=1,\ldots,n;\, r=0,1,\ldots;\,\, k=1,\ldots,K)\end{equation}
where $f^1,\ldots,f^K$ are given functions of a~finite number of coordinates (\ref{t1.1}) and
\begin{equation}\label{t1.3}D_i=\frac{\partial}{\partial x_i}+\sum w^j_{Ii}\frac{\partial}{\partial w^j_I}\qquad (i=1,\ldots,n)
\end{equation}
the total derivatives. Briefly saying, we deal with the infinite prolongation of the system $f^1=\cdots=f^K=0.$ This is just the~\emph{external approach}, differential equations are firmly localized in the jet space.

The \emph{internal approach} should not be affected by the inclusion $\mathbf M\subset\mathbf M(m,n).$ Within the common jet theory, this goal can be achieved after lengthy procedure \cite{T1,T2} which is moreover highly obscured by the unpleasant fact that the infinite--order jet spaces were not yet characterized in coordinate--free terms and the totality of all automorphisms is unknown. However differential equations without any additional structure affected by the ambient space $\mathbf M(m,n)$ and considered only on $\mathbf M$ can be precisely described on a~few lines by using some abstract algebraical properties of the module $\Omega$ of contact forms
\begin{equation}\label{t1.4}
\omega^j_I=\mbox{d}w^j_I-\sum w^j_{Ii}\mbox{d}x_i \quad (j=1,\ldots,m;\,I=i_1\cdots i_r;\, r=0,1,\ldots\,)\end{equation}
restricted to $\mathbf M.$ We speak of \emph{diffiety} $\Omega.$ Expressively saying, diffieties represent the system of differential equations in the \emph{absolute sense}, i.e., without any preferred choice of dependent and independent variables. There is only one disadvantage of this alternative approach \cite{T3,T4}. Though it is quite simple and brief, already the primary concepts and the more the final achievements cannot be translated into the common language of jets in full generality.

Our aim is to discuss a~far--going generalization of the classical controllability of the underdetermined systems of \emph{ordinary} differential equations $\Omega=\Omega^1$, i.e., the existence of the first integrals $\Omega^0\subset\Omega$. \emph{Partial} differential equations $\Omega$ behave much more involved: we obtain even a~large sequence $\Omega^0\subset\Omega^1\subset\cdots\subset\Omega$ of simplifying equations $\Omega^k$ (more precisely: of diffieties $\Omega^k$) which are induced by $\Omega.$
The upper indice $k$ declares that the (formal) solution of diffiety $\Omega^k$ depends on the choice of a~certain number $\mu(\Omega^k)\geq 1$ of arbitrary functions of $k$ independent variables (on constants if $k=0$).

This fundamental achievement can be informally illustrated as follows.\\
\begin{picture}(130,130)(-10,0)
\qbezier(10,110)(80,116)(100,120)\qbezier(20,80)(80,86)(110,90)\qbezier(10,110)(15,110)(20,80)
	\qbezier(20,100)(20,95)(110,100)
	\qbezier(20,95)(50,105)(110,115)
	\qbezier(30,90)(50,100)(110,110)
\qbezier(10,85)(10,85)(20,86)\qbezier(20,55)(80,61)(110,65)\qbezier(10,85)(15,85)(20,55)
	\qbezier(20,75)(90,65)(110,70)
	\qbezier(25,65)(90,80)(110,80)
\put(110,90){$\left.\begin{array}{l}\\\\\\\\\\\end{array}\right\}$}\put(130,90){$\Omega^1$}
\qbezier(10,20)(80,32)(100,45)
\qbezier(20,10)(80,20)(100,30)
\put(130,25){$\Omega^0$}
\put(110,25){$\left.\begin{array}{l}\\\\\\\end{array}\right\}$}
\put(0,45){leaves}\put(85,50){solutions}
\multiput(40,25)(0,5){12}{\put(0,0){\line(0,1){3}}}\put(40,50){\vector(0,-1){3}}
\multiput(50,15)(0,5){9}{\put(0,0){\line(0,1){3}}}\put(50,40){\vector(0,-1){3}}
\put(50,-10){Figure 1a.}
\end{picture}\hspace{2cm}
\begin{picture}(130,130)(0,0)
\put(110,90){$\left.\begin{array}{l}\\\\\\\\\\\end{array}\right\}$}\put(130,90){$\Omega$}
\put(110,25){$\left.\begin{array}{l}\\\\\\\end{array}\right\}$}\put(130,25){$\Omega^k$}
\put(10,75){\line(0,1){7}}
\qbezier(10,75)(50,80)(100,78)
                                      \qbezier(10,82)(50,87)(97,85)
                                      \qbezier(100,78)(100,78)(97,85)
\put(30,90){\line(0,1){7}}\qbezier(30,90)(70,95)(120,93)
                                      \qbezier(30,97)(70,102)(117,100)
                                      \qbezier(120,93)(120,93)(117,100)
\put(30,100){\line(0,1){7}}\qbezier(30,100)(70,105)(116,103)
                                      \qbezier(30,107)(70,112)(113,110)
                                      \qbezier(116,103)(116,103)(113,110)
					\qbezier(30,30)(70,35)(108,30)
					\qbezier(30,30)(70,30)(113,20)
					\qbezier(108,30)(108,30)(113,20)
					\qbezier(10,18)(80,18)(108,15)
					\qbezier(10,18)(70,12)(113,6)
					\qbezier(108,15)(108,15)(113,6)
\multiput(10,18)(0,5){13}{\put(0,0){\line(0,1){3}}}\put(10,50){\vector(0,-1){3}}
\multiput(30,30)(0,5){16}{\put(0,0){\line(0,1){4}}}\put(30,60){\vector(0,-1){3}}
\multiput(60,30)(0,5){14}{\put(0,0){\line(0,1){4}}}\put(60,60){\vector(0,-1){3}}
\multiput(80,27)(0,5){17}{\put(0,0){\line(0,1){4}}}\put(80,65){\vector(0,-1){3}}
\put(50,-10){Figure 1b.}
\put(85,50){solutions}
\end{picture}\\
\begin{itemize}
\item[-]{The left--hand figure describes the classical case of \emph{underdetermined systems of ordinary differential equations} $\Omega^1.$ In the non--controllable case, the original space is fibered by leaves which consist of solutions and the projection into the natural factorspace identifies all solutions lying in a~leaf which provide a~\emph{determined system} $\Omega^0.$ }
\item[-]{The right--hand figure describes the generalization: The original initial data for the solutions of $\Omega$ projected on lower--dimensional data of $\Omega^k$ and therefore some solutions are identified after the projections. This is however a~very rough description of the result.}
\end{itemize}
On this occasion, let us mention quite other perspectives of the theory. As yet we discussed the controllability on the total space $\mathbf M,$ it is however easy to introduce the controllability on a~subspace of $\mathbf M,$ in particular \emph{along a~fixed solution of} $\Omega.$
In the particular case of one independent variable, we obtain the classical Mayer extremals of the calculus of variations \cite[Figure 2]{T5}.
It should be expected that analogous "multidimensional Mayer extremals" appear for the case of several independent variables and the classical calculus of variations will be reduced to the controllability concepts.
\section{Introduction}\label{sec2}
The present article continues \cite{T4,T5,T6} but it is in principle made selfcontained. No advanced technical tools are needed, we deal with vector fields and differential forms on $\mathbb C^\infty$--smooth manifolds together with elementary algebra. Though the main concepts are of the global nature, the article is devoted to the~\emph{local theory}. The definition domains are not specified, e.g., our notational convention for a~mapping $\mathbf m: \mathbf M\rightarrow\mathbf N$ between manifolds allows the definition domain of $\mathbf m$ to be an~open subset of $\mathbf M.$ We also tacitly postulate the existence of bases in various modules to appear.

Besides the occasional use of the finite--dimensional manifolds, we mainly deal with the \emph{smooth manifolds $\mathbf M$ modelled on} $\mathbb R^\infty$ \cite{T4,T5,T6}.
They are supplied with (local) coordinates $h^i:\mathbf M\rightarrow \mathbb R$ $(i=1,2,\ldots\,)$ together with the \emph{structural ring} $\mathcal F=\mathcal F(\mathbf M)$ (the abbreviation, if possible) of functions $f:\mathbf M\rightarrow\mathbb R$ locally expressible by $C^\infty$--smooth formula $f=F(h^1,\ldots,h^{m(f)}).$
Then the $\mathcal F$--module $\Phi=\Phi(\mathbf M)$ of differential forms $\varphi=\sum f^ i\mbox{d}g^ i$ (finite sum with $f^ i, g^ i\in\mathcal F$) and the dual $\mathcal F$--module $\mathcal T=\mathcal T(\mathbf M)$ of vector fields $Z$ immediately appear. We recall that vector fields are regarded as $\mathcal F$--linear functionals $Z: \Phi\rightarrow\mathcal F$ where we denote
\[\varphi(Z)=Z\rfloor\varphi\in\mathcal F,\ \mbox{d}f(Z)=Z\rfloor\mbox{d}f=Zf\in\mathcal F\qquad (\varphi\in\Phi,\,Z\in\mathcal T,\,f\in\mathcal F).\]
The exterior differential $\,\mbox{d}\,$ and the Lie derivative $\mathcal L_Z$  satisfying the rules
\[\mathcal L_Z\varphi=Z\rfloor\mbox{d}\varphi+\mbox{d}\varphi(Z),\ \mathcal L_ZY=[Z,Y]\qquad (Y,Z\in\mathcal T;\,\varphi\in\Phi)\]
does not need any comments.

We shall deal with various $\mathcal F$--submodules $\Omega\subset\Phi.$ Then $\mathcal H=\mathcal H(\Omega)\subset\mathcal T$ denotes the submodule of all vector fields $Z$ such that $\Omega(Z)=0.$ A~submodule $\Omega\subset\Phi$ is called \emph{flat} (or: \emph{satisfying the Frobenius condition}) if any of the (equivalent) requirements
\[\mbox{d}\Omega\equiv 0\ (\mbox{mod }\Omega),\ \mathcal L_\mathcal H\Omega\subset\Omega,\ [\mathcal H,\mathcal H]\subset\mathcal H\quad (\mathcal H=\mathcal H(\Omega))\]
is satisfied.

\medskip

Our crucial concept appears if the classical flatness is completed with appropriate finiteness requirements as follows.
\begin{definition}\label{def2.1}
A~finite--codimensional submodule $\Omega\subset\Phi$ is called \emph{diffiety} if there exists a~filtration $\Omega_*: \Omega_0\subset\Omega_1\subset\cdots\subset\Omega=\cup\Omega_l$ with the finite--dimensional submodules $\Omega_l\subset\Omega$ $(l=0,1,\ldots\,)$ such that
\begin{equation}\label{t2.1}
\mathcal L_\mathcal H\Omega_l\subset\Omega_{l+1}\ (\mbox{all }l),\ \Omega_l+\mathcal L_\mathcal H\Omega_l=\Omega_{l+1}\ (l\mbox{ large enough}),
\end{equation}
the so called \emph{good filtration}.
\end{definition}

To the filtration $\Omega_*$ we introduce the graded \mbox{$\mathcal F$--module}
\begin{equation}\label{t2.2}
\mbox{Grad}\,\Omega_*=\mathcal M=\mathcal M_0\oplus\mathcal M_1\oplus\cdots\qquad (\mathcal M_l=\Omega_l/\Omega_{l-1},\, \Omega_{-1}=0).
\end{equation}
There are $\mathcal F$--linear mappings $Z:\mathcal M\rightarrow\mathcal M$ ($Z\in\mathcal H$) where
\[Z[\omega]=[\mathcal L_Z\omega]\in\mathcal M_{l+1}\qquad (\omega\in\Omega_l,\,[\omega]\in\mathcal M_l)\]
and the square brackets denote the factorization. However
\[\mathcal L_X\mathcal L_Y-\mathcal L_Y\mathcal L_X=\mathcal L_{[X,Y]}\qquad (X,Y,[X,Y]\in\mathcal H)\]
and we infer that $\mathcal M$ becomes a~graded $\mathcal A$--module where
\[\mathcal A=\mathcal A_0\oplus\mathcal A_1\oplus\cdots\qquad (\mathcal A_0=\mathcal F,\,\mathcal A_1=\mathcal H,\,\mathcal A_2=\mathcal H\odot\mathcal H,\,\ldots\,)\]
is the graded polynomial algebra over $\mathcal H$ \cite{T4,T6}. The multiplication is defined by
\[Z_1\cdots Z_r[\omega]=[\mathcal L_{Z_1}\cdots\mathcal L_{Z_r}\omega]\in\mathcal M_{l+r}\qquad (Z_1,\ldots,Z_r\in\mathcal H;\,\omega\in\Omega_l;\,[\omega]\in\mathcal M_l).\]
Owing to (\ref{t2.1}), $\mathcal M$ becomes a~Noetherian $\mathcal A$--module and the classical commutative algebra can be applied. In particular
\begin{equation}\label{t2.3}
\dim\mathcal M_l=e_\nu\binom l\nu+\cdots + e_0\binom l0\qquad (\,l \text{ large enough})  \end{equation}
is the \emph{Hilbert polynomial} with integer coefficients. Assuming $e_\nu\neq 0,$ the values $\nu=\nu(\Omega),$ $\mu=\mu(\Omega)=e_\nu$ do not depend on the choice of the filtration \cite{T4} and in accordance with the theory of exterior differential systems \cite{T7,T8} we declare that the~\emph{solution of $\Omega$ depends on $\mu(\Omega)$ functions of $\nu(\Omega)+1$ variables}. Trivially is $\nu(\Omega)+1\leq \dim\mathcal H=\dim \Phi/\Omega.$

\begin{definition}\label{def2.2}
Let $\Omega\subset\Phi$ be a~diffiety and $n=n(\Omega)=\dim\Phi/\Omega=\dim\mathcal H.$ A~certain functions $x_1,\ldots,x_n\in\mathcal F$ are \emph{called independent variables} for $\Omega$ if differentials $\text{d}x_1,\ldots,\text{d}x_n$ are independent modulo $\Omega.$ The vector fields $D_1,\ldots,D_n\in\mathcal H$ defined by the properties
\[D_ix_i=1,\, D_ix_j=\Omega(D_i)=0\qquad (i,j=1,\ldots,n;\, i\neq j)\]
are \emph{called total derivatives} with respect to the independent variables $x_1,\ldots,x_n.$
\end{definition}
It follows that $D_1,\ldots,D_n$ is a~basis of $\mathcal H$ and every form $\varphi\in\Phi$ admits unique decomposition
\[\varphi=\varphi(D_1)\text{d}x_1+\cdots +\varphi(D_n)\text{d}x_n+\omega\qquad (\omega\in\Omega).\]
In particular
\begin{equation}\label{t2.4}
\omega_f=\text{d}f-D_1f\text{d}x_1-\cdots -D_nf\text{d}x_n\in\Omega\qquad (f\in\mathcal F)
\end{equation}
for the choice $\varphi=\text{d}f.$

On this occasion, the interrelation between diffieties and the classical theory can be clarified as follows. Let us recall the subspace $\mathbf M\subset\mathbf M(m,n)$ given by requirements (\ref{t1.2}). Vector fields (\ref{t1.3}) are tangent to $\mathbf M$ and therefore may be regarded as vector fields on $\mathbf M$ and then identified with vector fields $D_1,\ldots,D_n$ of Definition~\ref{def2.2}. Also the contact forms (\ref{t1.4}) restricted to $\mathbf M$ are identical with forms (\ref{t2.1}) where $f=w^j_I.$ We conclude that the infinite prolongation of the classical differential equations represented by the Pfaffian system $\omega_f=0$ $(f\in\mathcal F)$ is a~diffiety. Conversely, every diffiety may be identified with such a~prolongation after the (in principle arbitrary) choice of independent and dependent variables, we refer to \cite{T6} for a~short proof.

We return to the general theory. Let us recall that diffieties $\Omega\subset\Phi$ are flat submodules with \emph{additional finite--dimensional requirements.} Our next aim is \emph{to delete such finiteness assumptions}, that is, to prove that certain flat submodules $\mathcal R\subset\Phi$ may be regarded for diffieties after an~appropriate adjustements.

\medskip

Our crucial lemma is as follows.
\begin{lemma}\label{l2.1}
Let $\Omega\subset\Phi$ be a~diffiety and $\mathcal R\subset\Omega$ a~flat submodule. There exists a~finite--dimensional submodule $\Gamma\subset\mathcal R$ such that
\begin{equation}\label{t2.5}
\mathcal R=\Gamma+\mathcal L_\mathcal H\Gamma+\mathcal L^2_\mathcal H\Gamma+\cdots\quad (\mathcal H=\mathcal H(\Omega))
\end{equation}
and
\begin{equation}\label{t2.6}
\mathcal R_*: \mathcal R_0\subset\mathcal R_1\subset\cdots\subset\mathcal R=\cup\mathcal R_l\quad(\mathcal R_l=\Gamma+\cdots +\mathcal L^l_\mathcal H\Gamma,\, \mathcal H=\mathcal H(\mathcal R))
\end{equation} is a~good filtration.
\end{lemma}
\begin{proof}
We take a~good filtration $\Omega_*$ and put $\mathcal R_l=\Omega_l\cap\mathcal R.$ Then the graded module $\mathcal N=\mathcal N_0\oplus\mathcal N_2\oplus\cdots\,$ where
\[\mathcal N_l=\mathcal R_l/\mathcal R_{l-1}=\Omega_l\cap\mathcal R/\Omega_{l-1}\cap\mathcal R\cong(\Omega_l\cap\mathcal R+\Omega_{l-1})/\Omega_{l-1}\subset\Omega_l/\Omega_{l-1}=\mathcal M_l\]
(isomorphism) is in fact a~graded $\mathcal A$--submodule of $\mathcal M.$ The Hilbert basis theorem from commutative algebra applies. It follows that
\[\mathcal H\mathcal N_l=\mathcal N_{l+1}\quad\text{hence}\quad \mathcal R_l+\mathcal L_\mathcal H\mathcal R_l=\mathcal R_{l+1}\qquad (l\text{ large enough},\,\mathcal H=\mathcal H(\Omega)).\] This implies (\ref{t2.5}) if $\Gamma=\mathcal R_l$ ($l$ fixed and large enough). Moreover $\mathcal H(\Omega)\subset\mathcal H(\mathcal R)$ and obviously (\ref{t2.6}) holds true. \end{proof}
\begin{theorem}\label{th2.1}
Let $\Omega\subset\Phi$ be a~diffiety and $\mathcal R\subset\Omega$ a~flat submodule of a~finite codimension. Then $\mathcal R\subset\Omega$ is a~diffiety too. \end{theorem}
Theorem~\ref{th2.1} is a~trivial consequence of Lemma~\ref{l2.1}. Alas, the finite codimensionality assumption is rather restrictive for future needs. We shall soon see that it is satisfied in certain "economical underlying space" of the module $\mathcal R.$
\begin{remark}\label{rem1}
The finite--dimensional underlying spaces $\mathbf M$ will also appear in our reasonings. Then the diffieties $\Omega\subset\Phi(\mathbf M)$ simplify. We may choose the trivial filtration $\Omega_*: \Omega_0=\Omega_1=\cdots=\Omega$ and put $\nu=\nu(\Omega)=-1,$ $\mu(\Omega)=\dim\Omega.$ One can observe that $\Omega$ has a~basis consisting of certain total differentials $\text{d}f^1,\ldots,\text{d}f^\mu$ by applying the Frobenius theorem to $\Omega$ in the finite--dimensional space $\mathbf M.$ Even the curious  particular subcase $\Omega=\Phi$ and $n(\Omega)=0$ makes a~good sense and should not be completely ignored.
\end{remark}

\section{Morphisms and projections}\label{sec3}
A~\emph{morphism} $\mathbf m: \mathbf M\rightarrow\mathbf N$ \emph{between manifolds} is defined by the property $\mathbf m^*\mathcal F(\mathbf N)\subset\mathcal F(\mathbf M).$ In terms of coordinates $h^i: \mathbf M\rightarrow\mathbb R$ and $k^ i: \mathbf N\rightarrow\mathbb R$ $(i=1,2,\ldots\,),$ certain $C^\infty$--smooth formulae \[\mathbf m^*k^i=K^i(h^1,\ldots,h^{m(i)})\qquad (i=1,2,\ldots\,)\] hold true.
Such a~morphism is called a~\emph{projection} (or: $\mathbf m$ is a~\emph{fibration} of~$\mathbf M$ with \emph{basis} $\mathbf N,$ or: $\mathbf N$ is a~\emph{factorspace} of~$\mathbf M$) if the family of functions $\mathbf m^*k^1,\mathbf m^*k^2,\ldots\in\mathcal F(\mathbf M)$ can be completed by appropriate functions of $\mathcal F(\mathbf M)$ to provide certain coordinates on $\mathbf M.$ (We also recall the common global definition: projection $\mathbf m$ is a~surjective submersion.)

We shall mainly deal with projections here. Then we occasionally abbreviate and even identify $f=\mathbf m^*f$ and $\varphi=\mathbf m^*\varphi$ which is possible since $\mathbf m^*$ is injective mapping. In more detail, we admit that
\begin{equation}\label{t3.1}
\mathcal F(\mathbf N)=\mathbf m^*\mathcal F(\mathbf N)\subset\mathcal F(\mathbf M),\ \Phi(\mathbf N)=\mathbf m^*\Phi(\mathbf N)\subset\Phi(\mathbf M) \end{equation}
may be regarded as $\mathcal F(\mathbf N)$--\emph{submodules} as well, according to the context.
(At this place, we apologize for such ``identifications'' and ``inclusions''. They are not formally correct. On the other hand, this point of view clarifies some constructions to appear and simplifies the formulation of the final result, Theorem~\ref{th3.2}.)
We also recall the $\mathbf m$--\emph{projectable vector fields} $Z\in\mathcal T(\mathbf M).$ They can be identified with the projections
\[\mathbf m_*Z\in\mathcal T(\mathbf N)\qquad ((\mathbf m_*Z)f=Z\mathbf m^*f,\, f\in\mathcal F(\mathbf N))\]
only modulo \emph{vertical} vector fields $V\in\mathcal T(\mathbf M)$ defined by the property
\[\mathbf m_*V=0\in\mathcal T(\mathbf N)\qquad (V\mathbf m^*k^i=0;\, i=0,1,\ldots\,).\]
The projections $\mathbf m$ and $\mathbf m_*$ are surjective.

The \emph{morphism between the submodules} $\Omega\subset\Phi(\mathbf M)$ and $\Theta\subset\Phi(\mathbf N)$ is defined by $\mathbf m^*\mathcal F(\mathbf N)\subset\mathcal F(\mathbf M)$ and $\mathbf m^*\Theta\subset\Omega.$ Let the morphism $\mathbf m$ be moreover a~projection. Then $\Theta=\mathbf m^*\Theta\subset\Omega$ may be regarded as $\mathcal F(\mathbf N)$--submodule of $\mathcal F(\mathbf M)$--module~$\Omega$ and, assuming this identification,
\[ \tilde\Theta=\mathcal F(\mathbf M)\Theta=\mathcal F(\mathbf M)\mathbf m^*\Theta\subset\Omega\] denotes the relevant
$\mathcal F(\mathbf M)$--submodule of $\Omega$ appearing after the extension of the coefficient ring. Alternatively saying, the $\mathcal F(\mathbf N)$--submodule $\Theta\subset\Omega$ generates the $\mathcal F(\mathbf M)$--submodule $\tilde\Theta\subset\Omega$.

Let the submodule $\Theta\subset\Phi(\mathbf N)$ be moreover flat. Then
\[\text{d}\tilde\Theta=\text{d}\mathcal F(\mathbf M)\wedge\Theta+\mathcal F(\mathbf M)\text{d}\Theta\cong 0\qquad (\text{mod}\,\Theta) \text{ hence } (\text{mod}\,\tilde\Theta)\]
and $\tilde\Theta$ also is flat. If in particular $\Omega\subset\Phi(\mathbf M)$ is a~diffiety, then Lemma~\ref{l2.1} can be applied and we see that submodule $\tilde\Theta \subset\Phi(\mathbf M)$ is a~\emph{pre--diffiety} in the sense that it satisfies all requirements of Definition~\ref{def2.1} except for the finite codimension. In fact this is only a~seeming defect.
Though both submodules $\tilde\Theta\subset\Omega\subset\Phi(\mathbf M)$ and $\Theta\subset\Phi(\mathbf N)$ need not be diffieties, the module $\Theta$ can be ``improved'' to become a~diffiety.

In more detail. In theory to follow, the module $\Theta$ is of a~mere subsidiary nature with respect to $\tilde\Theta.$ So we start with a~certain given flat submodule $\mathcal R\subset\Omega$ and our aim is to determine a~``good'' module $\Theta\subset\Phi(\mathbf N)$ such that $\tilde\Theta=\mathcal R.$
\begin{lemma}\label{l3.1}
Let a~flat submodule $\mathcal R\subset\Phi(\mathbf M)$ admit a~good filtration $(\ref{t2.6})$ where $\Gamma\subset\Phi(\mathbf M)$ is a~finite--dimensional submodule. There exists a~projection $\mathbf m: \mathbf M\rightarrow\mathbf N$ on certain space $\mathbf N$ and diffiety $\Theta\subset\Phi(\mathbf N)$ such that $\mathcal R=\mathcal F(\mathbf M)\mathbf m^*\Theta.$
\end{lemma}
\begin{proof}
We take a~basis $\gamma^1,\ldots,\gamma^K$ of $\Gamma.$ Let $x_1,x_2,\ldots\in\mathcal F(\mathbf M)$ be functions such that differentials $\text{d}x_1,\text{d}x_2,\ldots$ provide a~basis of $\Phi(\mathbf M)/\mathcal R.$ Then vector fields $X_1,X_2,...\in\mathcal H(\mathcal R)$ defined by
\[X_ix_i=1,\ X_ix_j=\mathcal R(X_i)=0\qquad (i,j=1,2,\ldots;\, i\neq j)\]
provide a~basis of module $\mathcal H(\mathcal R).$ (We have the infinite number of "independent variables" for the module $\mathcal R$ here.) The forms
\begin{equation}\label{t3.3}
\gamma^k_I=\mathcal L_{X_{i_1}}\cdots\mathcal L_{X_{i_r}}\gamma^k\ \ (k=1,\ldots,K; i_1,\ldots,i_r=1,2,\ldots; r=0,1,\ldots \,)
\end{equation}
generate $\mathcal R.$ However forms $\gamma^1,\ldots,\gamma^K$ can be expressed in terms of a~finite number of coordinates $h^1,\ldots,h^R$ and it follows that only certain functions
\begin{equation}\label{t3.4}
h^r_I=X_{i_1}\cdots X_{i_r}h^r\qquad (r=1,\ldots,R; i_1,\ldots,i_r\leq C)
\end{equation}
are sufficient to express all forms (\ref{t3.3}). One can even suppose $X_ih^r=0$ $(r=1,\ldots,R)$ if $i>C.$

We are passing to the delicate part of the proof. Let $g^1,g^2,\ldots\in\mathcal F(\mathbf M)$ be a~largest functionally independent subset of the set of all functions (\ref{t3.4}). We introduce manifold $\mathbf N$ with coordinates $k^i:\mathbf N\rightarrow\mathbb R$ and projection $\mathbf m:\mathbf M\rightarrow\mathbf N$ defined by $\mathbf m^*k^i=g^i$ $(i=1,2,\ldots).$
There are forms $\vartheta^j_I\in\Phi(\mathbf N)$ and vector fields $Z_1,Z_2,\ldots\in\mathcal T(\mathbf N)$ such that
\[\mathbf m^*\vartheta^j_I=\gamma^j_I,\ \mathbf m_*X_i=Z_i\qquad (\text{all }j,i,I).\]
Then
\[\mathcal L_{Z_i}\vartheta^j_I=\vartheta^j_{Ii}.\ \vartheta^j_I(Z_i)=0\qquad (\text{all }j,i,I)\]
by direct verification. All forms $\vartheta^j_I\in\Phi(\mathbf N)$ generate a~submodule $\Theta\subset\Phi(\mathbf N)$ and vector fields $Z_1,Z_2,\ldots$ generate the submodule $\mathcal H(\Theta)\subset\mathcal T(\mathbf N).$ Clearly $Z_i=0$ $(i>C)$ and therefore $n(\Theta)=\dim\mathcal H(\Theta)$ is finite. Identifying $\vartheta^j_I=\mathbf m^*\vartheta^j_I=\gamma^j_I,$ the existence of a~good filtration $\Theta_*$ is obvious. It follows that $\Theta\subset\Phi(\mathbf N)$ is diffiety satisfying $\mathcal R=\mathcal F(\mathbf M)\Theta.$ \end{proof}
\begin{theorem}\label{th3.1}
Let $\Omega\subset\Phi(\mathbf M)$ be a~diffiety and $\mathcal R\subset\Omega$ a~flat submodule. There exists a~projection $\mathbf m:\mathbf M\rightarrow\mathbf N$ on appropriate space $\mathbf N$ and diffiety $\Theta\subset\Phi(\mathbf N)$ such that $\mathcal R=\mathcal F(\mathbf M)\mathbf m^*\Theta.$
\end{theorem}
\begin{proof} Lemma~\ref{l2.1} ensures that Lemma~\ref{l3.1} can be applied. \end{proof}
\begin{remark}\label{remark2}
The proof of Lemma~\ref{l3.1} is of the independent interest since it contains explicit construction of the space $\mathbf N$ and the diffiety $\Theta\subset\Phi(\mathbf N).$ In applications to follow, the module $\mathcal R$ will be a~submodule of a~diffiety $\Omega\subset\Phi(\mathbf M).$ Then the functions $x_1,\ldots,x_n$ $(n=n(\Omega))$ appearing in the proof may be chosen as the independent variables of $\Omega,$ however, the vector fields $X_1,\ldots,X_n$ \emph{differ} from the total derivatives $D_1,\ldots,D_n$ though $D_1,\ldots,D_n\in\mathcal H(\mathcal R).$ On this occasion, we recall (\ref{t2.5}): already the forms
\[\mathcal L_{D_{i_1}}\cdots\mathcal L_{D_{i_k}}\gamma^k\qquad (k=1,\ldots,K;\,i_1,\ldots,i_r=1,\ldots,n;\,r=0,1,\ldots\,)\]
generate $\mathcal R,$ however, this fact is of a~little use for the proof. The crucial family of functions (\ref{t3.4}) with vector fields $X_1,X_2,\ldots$ cannot be ignored.

\medskip

The space $\mathbf N$ and the projection $\mathbf m$ are not unique but this does not matter in practice. Roughly saying, the use of the projection $\mathbf m: \mathbf M\rightarrow\mathbf N$ lies in the reduction of the (possibly) infinite--dimension of $\Phi(\mathbf M)/\mathcal R$ to the finite dimension of $\Phi(\mathbf N)/\Theta$ by deleting certain "parasite variables" of module $\mathcal R$ lying in $\mathcal F(\mathbf M)$ to obtain the "economical" space $\mathcal F(\mathbf N)$ and the "economical" module $\Theta.$
In fact there exists a~unique "minimal" space $\mathbf N$ without any "parasite variables", see Appendix. It is to be noted that all such calculations need a~bit of good luck, see the example of Section~\cite{sec6} below.
\end{remark}
\begin{remark}\label{rem2}
In the particular case of \emph{finite--dimensional} flat module $\mathcal R,$ much easier approach is possible. Due to Frobenius theorem, $\mathcal R$ does admit a~basis $\text{d}f^1,\ldots,\text{d}f^\mu$ consisting of total differentials. Functions $f^1,\ldots,f^\mu$ can be expressed in terms of finite number of coordinates $h^1,\ldots,h^R$ and then even the finite--dimensional space $\mathbf N$ with coordinates $k^1,\ldots,k^R$ and the projection $\mathbf m:\mathbf M\rightarrow\mathbf N$ defined by $\mathbf m^*k^i=h^i$ $(i=1,\ldots,R)$ resolve the problem. One can even choose $\mu$--dimensional space $\mathbf N$ with coordinates $h^1,\ldots,h^\mu$ and the projection $\mathbf m^*k^i=f^i$ $(i=1,\ldots,\mu).$ Then "the most economical" and "curious" diffiety $\Theta=\Phi(\mathbf N)$ mentioned in Remark~\ref{rem1} appears.
\end{remark}
\begin{definition}\label{def3.1}
Let $\Omega\subset\Phi(\mathbf M)$ be a~diffiety. Submodule $\mathcal R^k\subset\Omega$ of all forms $\omega\in\Omega$ such that
\begin{equation}\label{t3.5}\begin{array}{c}
\dim\{\omega,\mathcal L_\mathcal H\omega,\ldots,\mathcal L^l_\mathcal H\omega\}\leq c_k\binom lk+\cdots +c_0\binom l0\\
(l=0,1,\ldots; \mathcal H=\mathcal H(\Omega))\end{array}
\end{equation}
for appropriate integers $c_0=c_0(\omega),\ldots,c_k=c_k(\omega)$ is called $k$--th \emph{residual submodule} of $\Omega.$
\end{definition}There are inclusions
\begin{equation}\label{t3.6}
\mathcal R^0\subset\mathcal R^1\subset\cdots\subset\mathcal R^{\nu+1}=\Omega\qquad (\nu=\nu(\Omega)),
\end{equation}
however, we will omit the terms $\mathcal R^k$ such that $\mathcal R^k=\mathcal R^{k-1}$ in order to obtain only the proper inclusions in (\ref{t3.6}).
Every module $\mathcal R^k$ is flat, see below. Therefore Theorem~\ref{th3.1} ensures a~diffiety $\Theta=\Omega^k\subset\Phi(\mathbf N)$ and the projection $\mathbf m(k):\mathbf M\rightarrow\mathbf N$ such that
\begin{equation}\label{t3.7}
\mathcal R^k=\mathcal F(\mathbf M)\mathbf m(k)^*\Omega^k=\mathcal F(\mathbf M)\Omega^k.
\end{equation}
We denote $\mathbf M^k=\mathbf N$ for the better clarity from now on. The integers $c_0,\ldots,c_k$ in fact do not depend on the choice of the form $\omega$ and it follows that the solutions of diffiety $\Omega^k\subset\Phi(\mathbf M^k)$
depend on $c_k$ functions of $k$ variables, see below. For the convenience, we identify even $\mathcal R^k\cong\Omega^k$ and then the final achievement reads:
\begin{theorem}\label{th3.2}
Every diffiety admits a~unique composition series
\begin{equation}\label{t3.8}
\Omega^0\subset\Omega^1\subset\cdots\subset\Omega=\Omega^{\nu+1}\qquad (\nu=\nu(\Omega))
\end{equation}
(some terms may be absent) where $\Omega^k\subset\Phi(\mathbf M^k)$ is the maximal diffiety induced by $\Omega$ such that solution of $\Omega^k$ depends on $c_k>0$ functions of $k$ independent variables.
\end{theorem}
Before passing to examples, we return to the general theory of residual modules $\mathcal R^k$ since the original Definition~\ref{def3.1} obscures their position in diffiety $\Omega$ and is useless in practice.
\section{On the residual submodules}\label{sec4}
The submodules $\mathcal R^k\subset\Omega$ of diffiety $\Omega\subset\Phi\,(=\Phi(\mathbf M))$ deserve systematical discussion.
For the better clarity, we survey the preparatory concepts.
\subsection{Orthogonal submodules}\label{ss4.1}
Let $\Theta\subset\Phi$ be a~submodule. Then $\mathcal H(\Theta)\subset\mathcal T\,(=\mathcal T(\mathbf M))$ is the submodule of all vector fields $X$ satisfying $\Theta(X)=0.$ Modules $\Theta$ and $\mathcal H(\Theta)$ determine each other. For every vector field $Z\in\mathcal T,$ the inclusions $\mathcal L_Z\Theta\subset\Theta$ and $\mathcal L_Z\mathcal H(\Theta)\subset\mathcal H(\Theta)$ are equivalent. We abbreviate $\mathcal H=\mathcal H(\Omega)$ for the fixed diffiety $\Omega$ under consideration.
\subsection{Adjoint submodules}\label{ss4.2}
Let $\Theta\subset\Phi$ be a~submodule. Then $\text{Adj}\,\Theta\subset\Phi$ is the submodule generated by all forms
\[\vartheta,\mathcal L_X\vartheta=X\rfloor \text{d}\vartheta\qquad (\vartheta\in\Theta, X\in\mathcal H(\Theta)).\]
In alternative definition \cite{T4}, submodule $\mathcal H(\text{Adj}\,\Theta)\subset\mathcal T$ involves all vector fields $Y$ satisfying $\mathcal L_{fY}\Theta\subset\Theta$ for all $f\in\mathcal F\,(=\mathcal F(\mathbf M)),$ see the Appendix. It follows easily that $\text{Adj}\,\Theta$ is flat and for every $Z\in\mathcal T,$ the inclusion $\mathcal L_Z\Theta\subset\Theta$ implies $\mathcal L_Z\text{Adj}\,\Theta\subset\text{Adj}\,\Theta.$ Trivially $\Theta\subset\text{Adj}\,\Theta.$
\subsection{Kernel submodules}\label{ss4.3}
Let $\Theta\subset\Phi$ be a~submodule and $X\in\mathcal H(\Theta).$ Then $\text{Ker}_X\Theta\subset\Theta$ is the submodule of forms $\vartheta$ satisfying $\mathcal L_X\vartheta=X\rfloor\text{d}\vartheta\in\Theta.$ If $\Theta\subset\Omega$ is a~submodule of diffiety $\Omega,$ the submodules $\text{Ker}_X\Theta\subset\Theta$ $(X\in\mathcal H)$ make a~good sense.
\subsection{Hilbert polynomials}\label{ss4.4}
Equation (\ref{t2.3}) together with $\dim\Omega_l=\dim\mathcal M_0+\cdots +\dim\mathcal M_l$ implies that
\[\dim\Omega_l=c_{\nu+1}\binom l{\nu+1}+\cdots+c_0\binom l0\quad (l\text{ large enough}, c_{\nu+1}=e_\nu).\]
This is a~mere alternative (and intuitively better) transcription of~(\ref{t2.3}).
We recall that the solution of $\Omega$ depends on $c_{\nu+1}=\mu(\Omega)\geq 1$ functions of $\nu+1$ variables.
We suppose $\nu=\nu(\Omega)\geq 0$ here. If $\mathbf M$ is of a~finite dimension then $\dim\Omega_l=\dim\Omega=c_0=\mu(\Omega)$ for $l$ large enough.
\subsection{Other filtrations}\label{ss4.5}
To the primary filtration $\Omega_*,$ we introduce filtrations
\[\begin{array}{ll}
\Omega(Z_1)_*:\Omega(Z_1)_0\subset\Omega(Z_1)_1\subset\cdots\subset\Omega& (\Omega(Z_1)_l=\sum\mathcal L^k_{Z_1}\Omega_l),\\
\Omega(Z_1,Z_2)_*:\Omega(Z_1,Z_2)_0\subset\Omega(Z_1,Z_2)_1\subset\cdots\subset\Omega&(\Omega(Z_1,Z_2)_l=\sum\mathcal L^k_{Z_2}\Omega(Z_1)_l),\\
\cdots&\end{array}\]
of the same diffiety $\Omega$ where $Z_1,Z_2,\ldots\in\mathcal H$ are vector fields. The inclusions
\[\mathcal L_{Z_i}\Omega(Z_1,\ldots,Z_r)_l\subset\Omega(Z_1,\ldots,Z_r)_l\qquad (i=1,\ldots,r;\,l=0,1,\ldots)\]
are trivial.
\subsection{Not too special vector fields}\label{ss4.6}
Such vector fields $Z_1,\ldots,Z_n\in\mathcal H$ $(n=n(\Omega))$ are defined by the properties
\[\text{Ker}_{Z_1}\Omega_{l+1}=\Omega_l,\text{Ker}_{Z_2}\Omega(Z_1)_{l+1}=\Omega(Z_1)_l,\ldots\quad (l\text{ large enough}).\]Though the existence is nontrivial in full generality \cite{T6}, particular examples do not cause any difficulty. The practical rule is as follows: the modules $\text{Ker}_{Z_i}$ should be of the minimal possible dimension and this property survives small perturbations. One can even employ the total derivatives $Z_i=D_i$ for ``not too special'' choice of the independent variables.
\section*{The submodule $\mathcal R^0\subset\Omega$}
\emph{The existence.}
For $l\geq 0$ and $Z_1\in\mathcal H$ fixed, the series of proper inclusions
\begin{equation}\label{t4.1}\Theta\supset\text{Ker}_{Z_1}\Theta\supset\text{Ker}_{Z_1}^2\Theta\supset\cdots\qquad (\Theta=\Omega_l)\end{equation}
is finite. Indeed, the sequence
\[\dim\Theta\geq\dim\text{Ker}_{Z_1}\Theta\geq\dim\text{Ker}_{Z_1}^2\Theta\geq\cdots \]
becomes stationary at the finite length. Denoting
\begin{equation}\label{t4.2}\mathcal R^0(l)=\text{Ker}_{Z_1}^k\Theta=\text{Ker}_{Z_1}^k\Omega_l\qquad (k \text{ large enough}) \end{equation}
for this moment, then $\mathcal R^0(l)\subset\Omega_l$ is the largest submodule satisfying the inclusion $\mathcal L_{Z_1}\mathcal R^0(l)\subset\mathcal R^0(l).$

\medskip
\emph{The uniqueness.}
Let $Z_1$ be not too special from now on. Then the module $\mathcal R^0(l)=\mathcal R^0$ does not depend on $l$ if this $l$ is large enough. Moreover a~form $\omega\in\Omega$ (hence $\omega\in\Omega_l$ with $l$ large enough) lies in $\mathcal R^0$ if and only if
\begin{equation}\label{t4.3}\dim\{\omega,\mathcal L_{Z_1}\omega,\ldots,\mathcal L_{Z_1}^k\Omega\}\leq C_0\end{equation}
(where $C_0=\dim\mathcal R^0$) as follows by direct inspection. Due to criterion (\ref{t4.3}), the choice of the original filtration $\Omega_*$ is irrelevant.

\medskip
\emph{The universality.}
If $l$ is large enough, the sequence (\ref{t4.1}) hence the result $\mathcal R^0$ does not depend on the choice of the vector field $Z_1.$ We refer to \cite{T6} for a~direct elementary proof. It follows that $\mathcal L_\mathcal H\mathcal R^0\subset\mathcal R^0.$
Moreover a~form $\varphi\in\Phi$ is lying in $\mathcal R^0$ if and only if
\begin{equation}\label{t4.4}\dim\{\varphi,\mathcal L_\mathcal H\varphi,\ldots,\mathcal L_\mathcal H^k\varphi\}\leq C_0.\end{equation}
Indeed, condition (\ref{t4.4}) (more precisely: already condition $\dim\{\mathcal L_\mathcal H\varphi\}\leq C_0$) implies $\varphi=\omega\in\Omega$ and then (\ref{t4.3}) may be applied. Clearly $\mathcal R^0\subset\Omega$ (equivalently: $\mathcal R^0\subset\Phi$) is the largest finite--dimensional module (equivalently: module contained in $\Omega_l$ with $l$ large enough) such that $\mathcal L_\mathcal H\mathcal R^0\subset\mathcal R^0.$

\medskip
\emph{The Adj module.}
Let the forms $\rho^1,\ldots,\rho^R$ generate module $\mathcal R^0.$ Then the forms
\[\rho^r,\mathcal L_X\rho^r=X\rfloor\text{d}\rho^r\qquad (r=1,\ldots, R;\,X\in\mathcal H(\mathcal R^0))\]
generate $\text{Adj}\,\mathcal R^0$ which is therefore a~finite--dimensional module.

\medskip
\emph{The flatness.}
We recall the inclusion $\mathcal L_\mathcal H\mathcal R^0\subset\mathcal R^0$ which implies $\mathcal L_\mathcal H\text{Adj}\,\mathcal R^0\subset\text{Adj}\,\mathcal R^0$ hence $\text{Adj}\,\mathcal R^0\subset\Omega$ and therefore $\text{Adj}\,\mathcal R^0\subset\Omega_l$ if $l$ is large enough. Due to the maximality property of $\mathcal R^0,$ we conclude that $\text{Adj}\,\mathcal R^0\subset\mathcal R^0$ hence $\text{Adj}\,\mathcal R^0=\mathcal R^0$ is flat.
\section*{The submodule $\mathcal R^1\subset\Omega$}
The reasonings will be repeated with the only technical change: instead of the finite--dimensional \emph{modules} like $\Omega_l, \text{Ker}_Z\Omega_l, \mathcal R^0(l), \mathcal R^0, \text{Adj}\,\mathcal R^0$ and estimates like $\dim\{\cdots\}\leq const.,$ we deal with the \emph{filtrations} estimated by the first--order Hilbert polynomials. Expressively saying, operator $\mathcal L_{Z_1}$ was applied to \emph{forms} $\omega\in\Omega$ in the case of module $\mathcal R^0$ and we will apply operator $\mathcal L_{Z_2}$ to the \emph{strings} $\omega,\mathcal L_{Z_1}\Omega,\mathcal L_{Z_1}^2\Omega,\ldots$ to obtain module $\mathcal R^1.$

\medskip
\emph{The existence.}
For $l\geq 0$ and $Z_2\in\mathcal H$ fixed, the series of proper inclusions
\begin{equation}\label{t4.5}\Theta\supset\text{Ker}_{Z_2}\Theta\supset\text{Ker}_{Z_2}^2\Theta\supset\cdots\qquad (\Theta=\Omega(Z_1)_l)\end{equation}
is finite, see below.  Denoting
\begin{equation}\label{t4.6}\mathcal R^1(l)=\text{Ker}_{Z_2}^k\Theta=\text{Ker}_{Z_2}^k\Omega(Z_1)_l, \end{equation}
then $\mathcal R^1(l)\subset\Omega(Z_1)_l$ is the largest submodule with $\mathcal L_{Z_i}\mathcal R^1(l)\subset\mathcal R^1(l)$ $(i=1,2).$

Let us deal with the stationarity of (\ref{t4.5}). We introduce filtration
\[\Theta_*: \Theta_0\subset\Theta_1\subset\cdots\subset\Theta\qquad (\Theta_r=\Omega_l+\mathcal L_{Z_1}\Omega_l+\cdots +\mathcal L_{Z_1}^k\Omega_l)\]
where \[\dim\Theta_r=c_1\binom r1+c_0\binom r0=c_1r+c_0\qquad (r \text{ large enough}).\]
In more generality, we have filtrations
\[\Theta_*^k: \Theta_0^k\subset\Theta_1^k\subset\cdots\subset\Theta^k=\text{Ker}_{Z_2}^k\Theta\qquad (\Theta_r^k=\Theta_r\cap\text{Ker}_{Z_2}^k\Theta),\]
where
\[\dim\Theta_r^k=c_1^kr+c_0^k\ (r\text{ large enough}),\  c_1=c_1^0\geq c^1_1\geq c^2_1\geq\cdots\,.\]
Then $c_1^k=c_1^{k+1}=\cdots\ $ for $k$ large enough and this ensures the desired stationarity.

\medskip
\emph{The uniqueness.}
Let $Z_2$ be not too special from now on. Then $\mathcal R^1(l)=\mathcal R^1$ is independent of $l$ if $l$ is large enough. Moreover a~form $\omega\in\Omega$ lies in $\mathcal R^1$ if and only if
\begin{equation}\label{t4.7}\dim\{\omega,\ldots,\mathcal L_{Z_1}^r\mathcal L_{Z_2}^s\omega,\ldots,\mathcal L_{Z_1}^k\Omega,\mathcal L_{Z_2}^k\Omega\}\leq C_1k+C_0\qquad (r+s\leq k)\end{equation}
for certain integers $C_1$ and $C_0.$ (This again follows by a~direct inspection. The integers $C_1$ and $C_0$ are determined by obvious filtration of module $\mathcal R^1,$ in particular $C_1=c_1^r$ for $r$ large enough.) It follows that the choice of the original filtration $\Omega_*$ is irrelevant.

\medskip
\emph{The universality.}
If $l$ is large enough, the sequence (\ref{t4.5}) hence the result $\mathcal R^1$ does not depend on the choice of the vector field $Z_2,$ see \cite{T6}. It follows that $\mathcal L_\mathcal H\mathcal R^1\subset\mathcal R^1.$
Moreover a~form $\varphi\in\Phi$ is lying in $\mathcal R^1$ if and only if
\begin{equation}\label{t4.8}\dim\{\varphi,\mathcal L_\mathcal H\varphi,\ldots,\mathcal L_\mathcal H^k\varphi\}\leq C_1k+C_0.\end{equation}
Clearly $\mathcal R^1$ is the largest module which is contained in $\Omega(Z_1)_l$ with $l$ large enough and satisfying $\mathcal L_\mathcal H\mathcal R^1\subset\mathcal R^1.$

\medskip
\emph{The Adj module.}
Let the forms 
\begin{equation}\label{t4.9}\rho^r,\mathcal L_{Z_1}\rho^r,\mathcal L_{Z_1}^2\rho^r,\ldots\qquad (r=1,\ldots, R)\end{equation}
generate module $\mathcal R^1.$ Then the forms (\ref{t4.9}) together with all forms
\[X\rfloor\text{d}\rho^r,X\rfloor\text{d}\mathcal L_{Z_1}\rho^r,X\rfloor\text{d}\mathcal L_{Z_1}^2\rho^r,
\qquad (r=1,\ldots, R;\,X\in\mathcal H(\mathcal R^1))\]
generate module $\text{Adj}\,\mathcal R^1.$ Assuming
\[\begin{array}{ll}
\text{d}\rho^r\cong\sum a^r_{ij}\varphi_i\wedge\varphi_j&(\text{mod }\mathcal R^1),\\
\text{d}\mathcal L_{Z_1}\rho^r=\mathcal L_{Z_1}\text{d}\rho^r\cong\sum Z_1a^r_{ij}\,\varphi_i\wedge\varphi_j+\sum a^r_{ij}\mathcal L_{Z_1}(\varphi_i\wedge\varphi_j)&(\text{mod }\mathcal R^1),\\
\cdots
\end{array}\]
(use $\mathcal L_{Z_1}\mathcal R^1\subset\mathcal R^1$) we conclude that module $\text{Adj}\,\mathcal R^1$ is involved in the module generated by the forms
\begin{equation}\label{t4.10}
\varphi_i,\varphi_j,\mathcal L_{Z_1}\varphi_i,\mathcal L_{Z_1}\varphi_j,\mathcal L_{Z_1}^2\varphi_i,\mathcal L_{Z_1}^2\varphi_j,\ldots\,.\end{equation}
On the other hand $\mathcal L_\mathcal H\mathcal R^1\subset\mathcal R^1$ hence $\mathcal L_\mathcal H\text{Adj}\,\mathcal R^1\subset\text{Adj}\,\mathcal R^1$ which implies that $\text{Adj}\,\mathcal R^1\subset\Omega.$ Altogether it follows that $\text{Adj}\,\mathcal R^1$ is contained in all modules $\Omega(Z_1)_l$ if $l$ is large enough.

\medskip
\emph{The flatness.}
The above inclusion  $\mathcal L_\mathcal H\text{Adj}\,\mathcal R^1\subset\text{Adj}\,\mathcal R^1$ implies $\text{Adj}\,\mathcal R^1\subset\mathcal R^1$ hence $\text{Adj}\,\mathcal R^1=\mathcal R^1$ and this is a~flat module.

\section*{Remaining submodules $\mathcal R^k\subset\Omega$}
The above reasonings can be again literally repeated with the only change that the higher--order Hilbert polynomials estimating the filtrations appear. One can prove that the construction becomes trivial if $k\geq\nu(\Omega)$ since
\[\Omega(Z_1,\ldots,Z_\nu)_l=\Omega(Z_1,\ldots,Z_{\nu+1})_l=\cdots =\Omega\qquad (l \text{ large enough}, \nu=\nu(\Omega))\]
for every not too special sequence $Z_1,Z_2,\ldots\in\mathcal H$ and $\mathcal R^\nu=\mathcal R^{\nu+1}=\cdots=\Omega.$

\section{Example: ordinary differential equations}\label{sec5}
In the particular case $n=n(\Omega)=1$ of one independent variable, the controllability of Pfaffian systems in finite--dimensional spaces can be thoroughly described in terms of the Lie brackets $[X,Y]$ where $X,Y$ are vector fields satisfying the Pfaffian system, see \cite{T9} and references therein. Our approach is quite other.

Let us deal with the system
\[\frac{du}{dx}=F(x,u,v,w,\frac{dw}{dx}), \frac{dv}{dx}=G(x,u,v,w,\frac{dw}{dx})\quad (u=u(x), v=v(x), w=w(x)).\]
The corresponding diffiety $\Omega$ describes the infinite prolongation
\[\frac{du}{dx}=F(x,u,v,w_0,w_1), \frac{dv}{dx}=G(x,u,v,w_0,w_1), \frac{dw_r}{dx}=w_{r+1}\quad (r=0,1,\ldots)\]
of the system.

We introduce the space $\mathbf M$ with coordinates $x,u,v,w_0,w_1,\ldots\,,$ the submodule $\Omega\subset\Phi(\mathbf M)$ with the basis
\[\alpha=du-Fdx,\ \beta=dv-Gdx,\ \gamma_r=dw_r-w_{r+1}dx\quad (r=0,1,\ldots)\]
and the vector field
\[X=\frac{\partial}{\partial x}+F\frac{\partial}{\partial u}+G\frac{\partial}{\partial v}+\sum w_{r+1}\frac{\partial}{\partial w_r}\]
which is a~basis of $\mathcal H(\Omega).$ Clearly
\[\mathcal L_X\alpha=F_u\alpha+F_v\beta+F_{w_0}\gamma_0+F_{w_1}\gamma_1,\ \mathcal L_X\beta=G_u\alpha+G_v\beta+G_{w_0}\gamma_0+G_{w_1}\gamma_1\]
and $\mathcal L_X\gamma_r=\gamma_{r+1}$ $(r=0,1,\ldots).$ If $\Omega_l\subset\Omega$ $(l=0,1,\ldots)$ is the submodule generated by $\alpha,\beta,\gamma_0,\ldots,\gamma_l$ then $\Omega_*:\Omega_0\subset\Omega_1\subset\cdots$ is a~good filtration. Therefore $\Omega$ is a~diffiety with $n=n(\Omega)=\dim\mathcal H(\Omega)=1.$ Moreover $x_1=x$ is independent variable and $D_1=X$ the total derivative.

Clearly $\text{Ker}_X\Omega_{l+1}=\Omega_l$ $(l\geq 0)$ and the forms
\[\xi=\alpha-F_{w_1}\gamma_0,\ \zeta=\beta-G_{w_1}\gamma_0\] generate the submodule $\text{Ker}_X\Omega_0\subset\Omega_0.$ Passing to the next submodule, we state the formulae
\[\begin{array}{c}\mathcal L_X\xi=F_u\alpha+F_v\beta+(F_{w_0}-XF_{w_1})\gamma_0=F_u\xi+F_v\zeta+P\gamma_0,\\
\mathcal L_X\zeta=G_u\alpha+G_v\beta+(G_{w_0}-XG_{w_1})\gamma_0=G_u\xi+G_v\zeta+Q\gamma_0\end{array}\]
where
\[P=F_uF_{w_1}+F_vG_{w_1}+F_{w_0}-XF_{w_1},\ Q=G_uF_{w_1}+G_vG_{w_1}+G_{w_0}-XG_{w_1}.\]

Three subcases should be distinguished.

\medskip

If $P=Q=0$ identically then $\mathcal R^0=\text{Ker}_X\Omega_0.$ The general theory ensures that $\mathcal R^0$ is flat hence has a~certain alternative basis \[dU, dV\qquad (U=U(x,u,v,w_0), V=V(x,u,v,w_0)).\] So we have diffiety $\Omega^0\subset\Phi(\mathbf N)$ in the space $\mathbf N$ with coordinates $x,U,V.$ It corresponds to the determined system of differential equations \[\frac{dU}{dx}=0, \frac{dV}{dx}=0.\]
Quite explicit formulae for the functions $F,G,U,V$ in this subcase can be obtained but we omit details.

If either $P\neq 0$ or $Q\neq 0,$ then the form $\gamma=Q\xi-P\zeta$ generates module $\text{Ker}_X^2\Omega_0$ since
\[\mathcal L_X\gamma=XQ\,\xi- XP\,\zeta+Q(F_u\xi+F_v\zeta)-P(G_u\xi+G_v\zeta).\]
In general $\mathcal L_X\gamma$ is not a~multiple of $\gamma$ and then $\mathcal R^0=\text{Ker}_X^2\Omega_0=0$ is trivial.
Otherwise we obtain one--dimensional module $\mathcal R^0$ with the basis $\gamma.$ Since $\gamma$ is a~multiple of a~differential $dU$ $(U=U(x,u,v,w_0)),$ we obtain the diffiety $\Omega_0\subset\Phi(\mathbf N)$ in the space $\mathbf N$ with coordinates $x,U.$

In both above subcases, the space $\mathbf N$ naturally appears as a~factorspace of $\mathbf M.$ It is not the most economical one since variable $x$ may be in fact omitted and we obtain the ``curious diffiety'' of Remark~\ref{rem1}.

\section{Example: partial differential equations}\label{sec6}
While the residual module $\mathcal R^0$ is intuitively simple tradicional concept, the subsequent modules $\mathcal R^k$ $(k>0)$ are not so clear. Recall that they determine certain unique ``simplified projections'' of the original system of differential equations and essentially differ from the well--known reductions based on the Lie--group symmetries \cite{T12}--\cite{T15}. We intent to clarify the above abstract theory by means of explicit example of the module $\mathcal R^1.$ A~somewhat unusual strain of reasonings should be expected.
\subsection{The differential equation} We introduce the equation
\[\frac{\partial v}{\partial y}=F(x,y,u,v,\frac{\partial u}{\partial x}, \frac{\partial v}{\partial x},\frac{\partial u}{\partial y})\qquad (u=u(x,y), v=v(x,y))\]
together with the prolongation
\[\frac{\partial v_0}{\partial y}=F(x_1,x_2,u_{00},v_0,u_{10},v_1,u_{01}),\]
\[\frac{\partial u_{rs}}{\partial x_1}=u_{r+1,s}, \frac{\partial v_r}{\partial x_1}=v_{r+1}, \frac{\partial u_{rs}}{\partial x_2}=u_{r,s+1}, \frac{\partial v_r}{\partial x_2}=\frac{d^r}{dx_1^r}F\quad (r,s=0,1,\ldots)\]
where the alternative notation is better adapted for the general theory.
\subsection{The corresponding diffiety} We introduce the space $\mathbf M$ with coordinates
\[x_1,x_2,u_{rs},v_r\qquad (r,s=0,1,\ldots),\]
the submodule $\Omega\subset\Phi(\mathbf M)$ with the basis
\[\alpha_{rs}=du_{rs}-u_{r+1,s}dx_1-u_{r,s+1}dx_2,\ \beta_r=dv_r-v_{r+1}dx_1-D_1^rFdx_2\]
and the vector fields
\[D_1=\frac{\partial}{\partial x_1}+\sum u_{r+1,s}\frac{\partial}{\partial u_{rs}}+\sum v_{r+1}\frac{\partial}{\partial v_r},
\ D_2=\frac{\partial}{\partial x_2}+\sum u_{r,s+1}\frac{\partial}{\partial u_{rs}}+\sum D_1^rF\frac{\partial}{\partial v_r}\]
which provide a~basis of module $\mathcal H(\Omega).$ Clearly
\begin{equation}\label{t6.1}
\begin{array}{l}
\mathcal L_{D_1}\alpha_{rs}=\alpha_{r+1,s},\ \mathcal L_{D_2}\alpha_{rs}=\alpha_{r,s+1},\ \mathcal L_{D_1}\beta_r=\beta_{r+1},\\ 
\mathcal L_{D_2}\beta_0=F_{u_{00}}\alpha_{00}+F_{v_0}\beta_0+F_{u_{10}}\alpha_{10}+F_{v_1}\beta_1+F_{u_{01}}\alpha_{01},\\
\mathcal L_{D_2}\beta_r=\mathcal L_{D_2}\mathcal L_{D_1}^r\beta_0=\mathcal L_{D_1}^r\mathcal L_{D_2}\beta_0=\mathcal L_{D_1}^r(F_{u_{00}}\alpha_{00}+\cdots +F_{u_{01}}\alpha_{01}). \end{array}\end{equation}
If $\Omega_l\subset\Omega$ $(l=0,1,\ldots)$ is the submodule generated by forms $\alpha_{rs}, \beta_r$ $(r+s\leq l, r\leq l)$ then we obtain a~good filtration $\Omega_*:\Omega_0\subset\Omega_1\subset\cdots$ of $\Omega.$ It follows that $\Omega$ is a~diffiety. Clearly $n=n(\Omega)=\dim\mathcal H(\Omega)=2,$ $x_1$ and $x_2$ are independent variables with  $D_1$ and $D_2$ the total derivatives. Moreover $\nu=\nu(\Omega)=1,$ $\mu=\mu(\Omega)=1.$
\subsection{The triviality of $\mathcal R^0$} Clearly
\[\text{Ker}_{D_1}\Omega_{l+1}=\Omega_l\quad (l\geq 0),\ \text{Ker}_{D_1}\Omega_0=0.\] The sequence (\ref{t4.1}) terminates with the trivial stationarity $\mathcal R^0=0,$ there do not exist first integrals.
\subsection{Towards the module $\mathcal R^1$} Recalling (\ref{t4.5}), we introduce the submodules $\Omega(D_1)_l\subset\Omega$ $(l=0,1,\ldots)$ with the basis
\[\alpha_{rs}=\mathcal L_{D_1}^r\alpha_{0s},\ \beta_r=\mathcal L_{D_1}^r\beta_0\qquad (r=0,1,\ldots;\, s\leq l).\]
Clearly
\[\text{Ker}_{D_2}\Omega(D_1)_{l+1}=\Omega(D_1)_l\qquad (l\geq 0)\]
but the case $l=0$ is more interesting. Using (\ref{t6.1}), one can infer that
\begin{equation}\label{t6.2}\begin{array}{c} \mathcal L_{D_2}(\beta_0-F_{u_{01}}\alpha_{00})=\\ (F_{u_{00}}-D_2F_{u_{01}})\alpha_{00}+F_{v_0}\beta_0+F_{u_{10}}\alpha_{10}+F_{v_1}\beta_1\in\Omega(D_1)_0 \end{array}\end{equation}
therefore
\[\gamma=\beta_0-F_{u_{01}}\alpha_{00}\in\text{Ker}_{D_2}\Omega(D_1)_0.\]
Then trivially
\[\mathcal L_{D_1}^r\gamma\in\Omega(D_1)_0,\ \mathcal L_{D_2}\mathcal L_{D_1}^r\gamma=\mathcal L_{D_1}^r\mathcal L_{D_2}\gamma\in\mathcal L_{D_1}^r\Omega(D_1)_0\subset\Omega(D_1)_0\]
and it follows that the forms
\begin{equation}\label{t6.3}\gamma_r=\mathcal L_{D_1}^r\gamma\qquad (r=0,1,\ldots;\,\gamma=\beta_0-F_{u_{01}}\alpha_{00})
\end{equation}
provide a~basis of module $\text{Ker}_{D_2}\Omega(D_1)_0.$ In order to determine the subsequent term $\text{Ker}_{D_2}^2\Omega(D_1)_0$ of sequence (\ref{t4.5}), we state the formulae
\begin{equation}\label{t6.4}\beta_0=\gamma+F_{u_{01}}\alpha_{00},\ \beta_1=\mathcal L_{D_1}\beta_0=\gamma_1+F_{u_{01}}\alpha_{10}+D_1F_{u_{01}}\alpha_{00}\end{equation}
whence the equation
\[\mathcal L_{D_2}\gamma=A\alpha_{00}+B\alpha_{10}+F_{v_0}\gamma+F_{v_1}\gamma_1,\]
with
\[A=F_{u_{00}}+F_{v_0}F_{u_{01}}+F_{v_1}D_1F_{u_{01}}-D_2F_{u_{01}},\ B=F_{u_{01}}+F_{v_1}F_{u_{01}}\]
follows by the substitution of (\ref{t6.4}) into (\ref{t6.2}).

\medskip

{\bf Summary 1.}\emph{ If either $A\neq 0$ or $B\neq 0$ then $\mathcal R^1=0.$ Otherwise
\[\mathcal R^1=\text{Ker}_{D_2}^2\Omega(D_1)_0=\text{Ker}_{D_2}\Omega(D_1)_0\]
is nontrivial module with the basis $(\ref{t6.3}).$}

\subsection{The existence problem} We are interested just in the noncontrollable case when $A=B=0$ from now on. In order to determine such diffieties, let us alternatively use the traditional notation
\[x=x_1,y=x_2,u=u_{00},u_x=u_{10},\ldots,\ v=v_0,v_x=v_1,v_{xx}=v_2,\ldots\]
and then the top--order summands of $A$ are
\begin{equation}\label{t6.5}\begin{array}{ll}
A=&\cdots +F_{v_x}(F_{u_yu_x}u_{xx}+F_{u_yv_x}v_{xx}++F_{u_yu_y}u_{xy})\\
&\quad -F_{u_yu_x}u_{xy}-F_{u_yv_x}(F_{u_x}u_{xx}+F_{v_x}v_{xx}+F_{u_y}u_{xy})-F_{u_yu_y}u_{yy}.
\end{array}\end{equation}
It follows that
\[F_{u_yu_y}=0, F=f(x,y,u,v,u_x,v_x)u_y+g(x,y,u,v,u_x,v_x)\]
and (\ref{t6.4}) vanishes if moreover
\[f_{u_x}g_{v_x}=f_{v_x}g_{u_x},\ f_{u_x}+ff_{v_x}=0.\]
Such requirements are satisfied if
\begin{equation}\label{t6.6} g=\bar g(x,y,u,v,f),\ H(x,y,u,v,f)+u_xf=v_x \end{equation}
where $\bar g, H$ may be arbitrary functions. Assuming (\ref{t6.6}), identity $B=0$ also is satisfied (direct verification). Finally, the lower--order terms in $A$ provide the concluding requirements
\begin{equation}\label{t6.7} f_{v_x}(\bar g_u+\bar g_vf)=0,\ \bar g_f(f_u+f_vf)=f_y+f_v\bar g+f_{v_x}(\bar g_x+\bar g_vH) \end{equation}
(direct verification). One can calculate the derivatives $f_{v_x},f_u,f_v,f_y$ by the implicit equation (\ref{t6.6}) and the requirements (\ref{t6.7}) turn into two equations
\begin{equation}\label{t6.8} \bar g_u+\bar g_vf=0,\ \bar g_f(H_u+H_vf)=H_y+H_v\bar g-\bar g_x-\bar g_vH \end{equation}
for two unknown functions $G,H.$
\subsection{A~simple controllable problem} We are not interested in complete discussion of equations (\ref{t6.7}) or (\ref{t6.8}) here. Let us therefore introduce the ``brutal solution'' which appears if
\[f_y=f_u=f_v=0\quad\text{hence}\quad H_y=H_u=H_v=0.\] Then the requirements (\ref{t6.7}) simplify as
\[\bar g_u+\bar g_vf=0,\ \bar g_x+\bar g_vH=0.\]
Assuming moreover $f_x=H_x=0,$ we obtain the solution $\bar g=G(y,Hx+fu-v)$ which is quite sufficient for our modest aim.

\medskip

{\bf Summary 2.}\emph{We have the noncontrollable case
\[F=fu_y+g,\ g=G(y,Hx+fu-v),\ H+u_xf=v_x,\ f=f(u_x,v_x)\]
where $G=G(y,w)$ and $H=H(f)$ may be arbitrary functions.}

\medskip

In order to avoid trivialities, we suppose $G_w\neq 0$ and $H'\neq 0.$
\subsection{Preparatory remarks} Let us recall the form
\[\begin{array}{rl}
\gamma=\beta_0-f\alpha_{00}&=dv-v_xdx-(fu_y+g)dy-f(du-u_xdx-u_ydy)\\ &=dv-fdu-Hdx-Gdy\in\mathcal R^1 \end{array}\]
and the basis
\begin{equation}\label{t6.9} \gamma_r=\mathcal L_{D_1}^r\gamma=dv_r-\mathcal L_{D_1}^r(fdu)-D_1^rHdx-D_1^rGdy\qquad (r=0,1,\ldots)\end{equation}
of module $\mathcal R^1$ where
\begin{equation}\label{t6.10} \mathcal L_{D_1}^r(fdu)=\binom r0fdu_{r0}+\binom r1 D_1fdu_{r-1,0}+\cdots +\binom rr D_1^rfdu_{00}. \end{equation}
The formulae are of the fundamental importance.
\subsection{Some complementary remarks} Clearly
\[d\gamma\cong (du+Hdx+G_w(H'x+u)dy)\wedge df\qquad (\text{mod }\gamma)\]
and we recall the classical Adj--module for the Pfaffian equation $\gamma=0.$ This is a~\emph{flat module} with the basis
\[\gamma,\ df,\ du+Hdx+G_w(H'x+u)dy\]
in the space of the variables $x,y,u,v,f.$ Due to the Frobenius theorem, there exists alternative basis $dM, dN, dP$ such that $\gamma=Q(dM-PdN)$ for appropriate factor $Q.$ We may suppose $N=f$ without loss of generality. It is worth mentioning that the congruence
\[d\gamma\cong dQ\wedge dM,\ d\gamma\cong dy\wedge G_w(Hdx+fdu-dv)=-Q_wdy\wedge\gamma\quad (\text{mod }df)\]
imply useful formula
\[\frac{dQ}{Q}\cong -G_wdy\ (\text{mod }df),\ \ln Q=-\int G_wdy+C(f,Hx+fu-v)\] for the factor $Q.$
\subsection{Toward the space $\mathbf N$}
Let us literally follow the proof of Lemma~\ref{l3.1}. The differentials
\[dx=dx_1, dy=dx_2, du_{rs}\qquad (r,s=0,1,\ldots)\]
clearly provide a~basis of module $\Phi(\mathbf M)/\mathcal R^1.$ We introduce the ``dual'' basis
\[X,Y,U_{rs}\in\mathcal R^1\qquad (r,s=0,1,\ldots)\]
defined by
\[\begin{array}{l}
Xx=Yy=1,\ Xy=Yx=Xu_{rs}=Yu_{rs}=0,\\
U_{rs}u_{rs}=1,\ U_{rs}x=U_{rs}y=U_{rs}u_{r's'}=0\qquad (r\neq r' \text{ or } s\neq s') \end{array}\]
and moreover the more interesting formulae
\[\begin{array}{c}
Xv_r=D_1^rH, Yv_r=D_1^rG, U_{r-k,0}v_r=\binom rk D_1^kf\quad (k=0,\ldots,r)\\
U_{r0}v_{r'}=0\quad (r>r'), U_{rs}v_{r'}=0\quad (s\neq 0) \end{array}\]
follow from (\ref{t6.9}) and (\ref{t6.10}).
\subsection{The common continuation}
The form $\gamma$ is expressible in terms of coordinates $x,y,u,v,u_x,v_x.$ It follows that vector fields
\[U_{r0}\quad (r>1),\ U_{rs}\quad (s>0)\]
do not affect the space $\mathbf N$ in the sense that the module $\mathcal R^1$ is generated by the forms
\begin{equation}\label{t6.11} \mathcal L_X^k\mathcal L_Y^l\mathcal L_{U_{00}}^s\mathcal L_{U_{10}}^r\gamma\qquad (k,l,r,s=0,1,\ldots)\end{equation}
which are expressible in terms of functions
\begin{equation}\label{t6.12} X^kY^lU_{00}^sU_{10}^rh\qquad (h=x,y,u,v,u_x,v_x;\, k,l,r,s=0,1,\ldots)\end{equation}
while the application of other vector fields produces only zero forms and identically vanishing functions. Briefly saying, functions (\ref{t6.12}) should be taken for coordinates on $\mathbf N=\mathbf M^1,$ the forms (\ref{t6.11}) generate the diffiety $\Theta=\Omega^1\subset\Phi(\mathbf N)$ and the (natural projection) of vector fields $X,Y,U_{00},U_{10}$ provide the basis of module $\mathcal H(\Theta)=\mathcal H(\Omega^1).$
\subsection{A~slightly better approach}
In fact the form $\gamma$ is expressible in term of the functions $x,y,u,v,f$ and the module $\mathcal R^1$ is generated by the forms
\begin{equation}\label{t6.13} \mathcal L_U\gamma=dU^rv-U^rfdu-U^rHdx-U^rGdy\qquad (U=U_{00};\,r=0,1,\ldots)\end{equation}
which are expressible in terms of the functions
\begin{equation}\label{t6.14} x,y,u,U^rv,U^rf\qquad (U=U_{00};\,r=0,1,\ldots).\end{equation}
(This follows from the inspection of the top--order terms:
\[Uv=f, U^{r+1}v=U^rf=f_{v_x}U^rv_x=f_{v_x}D_1^rf=\cdots +(f_{v_x})^2v_{r+1}\]
whence
\[\mathcal L_U^r\gamma=(f_{v_x})^2dv_r+\cdots\qquad (r=0,1,\ldots)\]
and we indeed have a~basis of $\mathcal R^1.$)
The functions (\ref{t6.14}) provide coordinates on $\mathbf N=\mathbf M^1,$ the forms (\ref{t6.13}) provide a~basis of diffiety $\Theta=\Omega^1\subset\Phi(\mathbf N)$ and the space $\mathcal H(\Theta)=\mathcal H(\Omega^1)$ is reduced since
\[U_{10}x=U_{10}y=U_{10}u=U_{10}v=U_{10}f=f_{u_x}+f_{v_x}f=0\]
and the vector field $U_{10}$ may be omitted.
\subsection{The reduced differential equation}
The diffiety $\Theta\subset\Phi(\mathbf N)$ is a~prolongation of the Pfaffian equation $\gamma=0$ and therefore corresponds to the system
\[\frac{\partial v}{\partial u}=f, \frac{\partial v}{\partial x}=H(f), \frac{\partial v}{\partial y}=G(y,H(f)x+fu-v)\]
which is equivalent to the system
\[\frac{\partial v}{\partial x}=H\left(\frac{\partial v}{\partial u}\right), \frac{\partial v}{\partial y}=G\left(y,H\left(\frac{\partial v}{\partial u}\right)+\frac{\partial v}{\partial u}u-v\right)\qquad (v=v(x,y,u)).\]
(We may also recall the above result: in fact we have a~Pfaffian equation \mbox{$dM-Pdf=0$} in certain ``the most economical space $\mathbf N$'' with the explicit solution $M=M(f), P=M'(f).$)

\section*{Appendix}
Let us first informally mention the well--known concept of the infinitesimal symmetry of a~''geometrical object $\mathcal A$'' on a~space $\mathbf M.$ Such infinitesimal symmetry $Z\in\mathcal T(\mathbf M)$ is defined by the property that the Lie derivative $\mathcal L_Z$ ``does not change~$\mathcal A$''. As a~result, there appear a~Lie algebra over $\mathbb R$ of such vector fields~$Z.$

A~slight change of this idea provides the Adj--module \cite{T4}. Let us suppose that even all Lie derivatives $\mathcal L_{fZ}$ $(f\in\mathcal F(\mathbf M))$ do not change $\mathcal A.$ Then the geometrical intuition suggests the idea that the object is ``represented by the orbits of $Z$''. Alternatively saying, $\mathcal A$ can be ``expressed in terms of functions $f\in\mathcal F(\mathbf M)$'' constant along the orbits. In other words, if $\text{Adj}\,\mathcal A\subset\Phi(\mathbf M)$ is the submodule generated by differentials $df$ then $\mathcal H(\text{Adj}\,\mathcal A)\subset\mathcal T(\mathbf M)$ is generated by vector fields $Z.$

\medskip

Examples. If $\mathcal A\subset\Phi(\mathbf M)$ is a~\emph{subset} of differential forms, vector fields $Z$ satisfy $\mathcal L_{fZ}\varphi=0$ $(\varphi\in\mathcal A).$ If $\mathcal A\subset\Phi(\mathbf M)$ is a~\emph{submodule}, we require $\mathcal L_{fZ}\mathcal A\subset\mathcal A.$ Instead of differential forms, we may take tensors as well. For the exterior systems, the Adj--module describes just the classical Cauchy characteristics.

\medskip

The Adj--modules frequently appear already in early E. Cartan's articles, see especially \cite{T7,T9} and we also refer to recent article \cite{T10} for quite other approach and useful review of classical literature. All these authors however deal with finite--dimensional spaces $\mathbf M.$ In our infinite--dimensional space $\mathbf M,$ certain causion is necessary since the vector fields $Z$ need not generate any group and therefore ``do not produce'' any orbits. In order to obtain ``economical variables for $\mathcal A$'', it is necessary to introduce the \emph{Cauchy submodule $\mathcal C$} of module Adj. On this occasion, we refer to the following result \cite[VII. 6]{T4}.
\\\\
{\bf Proposition.} \emph{Let $\Omega\subset\Phi(\mathbf M)$ be a~diffiety with a~good filtration $\Omega_*.$ Let $\mathcal C(\Omega)\subset\mathcal H(\Omega)$ be the submodule of all vector fields $Z$ such that $\mathcal L_Z^k\Omega_l\subset\Omega_{l+c(Z)}$ for all (equivalently: for some) $l$ large enough. Then there exists a~basis of $\Omega$ expressible in terms of functions $f\in\mathcal F(\mathbf M)$ such that $Zf=0$ $(Z\in\mathcal C(\Omega)).$}

\bigskip

Alternatively saying, the orbits of vector fields $Z\in\mathcal C(\Omega)$ exist and may be regarded for the ``absolut Cauchy characteristics'' of the diffiety $\Omega.$ The Proposition remains true for the pre--diffieties \cite[VIII. 3]{T4}. In this way, the uniquely determined underlying space $\mathbf N$ of flat submodules $\mathcal R\subset\Omega$ without any ``parasite variables'' appears.

\end{document}